\documentclass[letterpaper, 10 pt, journal, twoside]{IEEEtran}
\usepackage{cite}
\usepackage{amsmath,amssymb,amsfonts}
\usepackage{algorithmic}
\usepackage{graphicx}
\usepackage{textcomp}
\usepackage{amsthm}

\usepackage{tikz}
\usetikzlibrary{shapes,arrows}
\usepackage{nicefrac}
\usepackage{mathtools, cuted}
\usepackage{lipsum}
\usepackage{float}
\usepackage{afterpage}
\usepackage{placeins}
\usepackage{subcaption}
\usepackage{tabularx}
\usepackage{balance}
\usepackage{textcomp}%
\usepackage{booktabs}
\usepackage{multirow}
\usepackage{accents}
\usetikzlibrary{patterns,shapes,arrows}
\usetikzlibrary{positioning}
\allowdisplaybreaks

\newtheorem{proposition}{Proposition}%[theorem]

\newcommand\copyrighttext{%
  \footnotesize \textcopyright 2020 IEEE. Personal use of this material is permitted. Permission from IEEE must be obtained for all other uses, in any current or future media, including reprinting/republishing this material for advertising or promotional purposes, creating new collective works, for resale or redistribution to servers or lists, or reuse of any copyrighted component of this work in other works.}
\newcommand\copyrightnotice{%
\begin{tikzpicture}[remember picture,overlay]
\node[anchor=south,yshift=10pt] at (current page.south) {\fbox{\parbox{\dimexpr\textwidth-\fboxsep-\fboxrule\relax}{\copyrighttext}}};
\end{tikzpicture}%
}

\begin{document}
\title{Efficient Implementation of Rate Constraints \\for Nonlinear Optimal Control}

\author{Yuanbo Nie and Eric C. Kerrigan 
\thanks{Yuanbo Nie and Eric C. Kerrigan are with the Department of Aeronautics, Imperial College London, SW7~2AZ, U.K. {\tt\small yn15@ic.ac.uk}, {\tt\small 
e.kerrigan@imperial.ac.uk}}%
\thanks{Eric C. Kerrigan is also with the Department of Electrical \& Electronic Engineering, Imperial College London, London SW7~2AZ, U.K.}%
\thanks{Accepted version to be published in: IEEE Transactions on Automatic Control}%
}

\maketitle
\copyrightnotice

\begin{abstract}

We propose a general approach to directly implement rate constraints on the discretization mesh for all collocation methods, for both state and input variables. Unlike conventional approaches that may lead to  singular control arcs,  the solution of this on-mesh implementation has better  properties. Moreover, computational speedups of more than 30\%   can be achieved by exploiting the properties of the resulting linear constraint equations. 
\end{abstract}


\section{Introduction}
Optimization-based control strategies, such as model predictive control (MPC), can be seen in an increasing number of applications. For many  engineering problems,  constraints may need to be imposed on the rate of changes for the state and/or input variables, to account for physical actuation limitations (e.g.\ the maximum rotation rate of flight control surfaces on aircraft) or to fulfill certain ride comfort requirements (e.g.\ the maximum longitudinal and lateral accelerations experienced by passengers). 

In optimal control, the underlying optimization problem can often be formulated and implemented in a number of different ways. Under a linear framework, many implementations are computationally comparable, thus  straightforward  approaches are often used.  For example, rate constraints on input variables are generally implemented through additional dynamic equations  \cite{wang2009model, maciejowski2002predictive}, and rate constraints on state variables are commonly addressed with additional path constraints~\cite{deori2015model}. However, under a nonlinear framework, this way of implementing input rate constraints are known to result in numerical difficulties and introducing fluctuations and ringing phenomena in the solution due to singular control \cite{betts2010practical}.  To improve the solution quality, additional regularization terms may be added to the optimal control problem (OCP) formulation \cite{Bonnans2008}; however, this practice often leads to computational challenges by needing to solve the problem repetitively with appropriate weightings.

In our previous work, we proposed implementing rate constraints directly on the discretization mesh with linear constraints \cite{ratePaper}. In this paper, we present a more in-depth  analysis of this method. We demonstrate that the proposed method will not introduce singular arcs to the problem, resulting in solutions of higher accuracy than the conventional approach. The computational comparisons with the  conventional  implementation as well as the regularization approach are also notably expanded with further insights. 

Sections~\ref{sec:OptimizationBasedControl}--\ref{sec:DirectTranscriptionMethod} aim at providing a brief introduction on solving OCPs with direct collocation methods. Following this, different approaches for implementing rate constraints in the OCP are introduced and analysed in Section \ref{sec:AlgebraicVariableRateConstraints}. This is followed by two classical examples of different complexity in Section \ref{sec:ExampleProblems}, where the pros and cons of each implementation are demonstrated. In Section \ref{sec:Conclusion}, we provide concluding remarks and some guidelines for implementation. 

\section{Optimal Control Problem}
\label{sec:OptimizationBasedControl}
Generally speaking, optimization-based control requires the solution of  optimal control problems with the objective functional expressed in the general Bolza form:

\begin{subequations}
\label{eqn:OCPBolza}
\begin{equation}
%\begin{split}
\min_{x,u,p,t_0,t_f} \Phi(x(t_0),t_0,x(t_f),t_f,p)
+\int_{t_0}^{t_f} L(x(t),u(t),t,p) dt
%\end{split}
\end{equation}
subject to
\begin{align}
\dot{x}(t)=f(x(t),u(t),t,p),\ &\forall t \in [t_0,t_f] \label{eqn:OCPBolzaDynamics}\\
c(x(t),u(t),t,p)\le 0,\ &\forall t \in [t_0,t_f]\\
\phi(x(t_0),t_0,x(t_f),t_f,p) =0,\ &
\end{align}
\end{subequations}
 with $x: \mathbb{R} \rightarrow \mathbb{R}^n$ is the state trajectory of the system, $u: \mathbb{R} \rightarrow \mathbb{R}^m$ is the control input trajectory,   $p \in \mathbb{R}^s$ are static parameters, $t_0 \in \mathbb{R}$ and $t_f \in \mathbb{R}$ are the initial and terminal time.  $\Phi$ is the Mayer cost functional ($\Phi$: $\mathbb{R}^n \times \mathbb{R} \times \mathbb{R}^n \times \mathbb{R} \times \mathbb{R}^s \to \mathbb{R}$), $L$ is the Lagrange cost functional ($L:\mathbb{R}^n \times \mathbb{R}^m \times \mathbb{R} \times \mathbb{R}^s \to \mathbb{R}$), $f$ is the dynamic constraint ($f:\mathbb{R}^n \times \mathbb{R}^m \times \mathbb{R} \times \mathbb{R}^s \to \mathbb{R}^n$), $c$ is the path constraint ($c:\mathbb{R}^n \times \mathbb{R}^m \times \mathbb{R} \times \mathbb{R}^s \to \mathbb{R}^{n_g}$) and $\phi$ is the boundary condition ($\phi:\mathbb{R}^n \times \mathbb{R} \times \mathbb{R}^n \times \mathbb{R} \times \mathbb{R}^s \to \mathbb{R}^{n_q}$).

\section{Direct collocation method}
\label{sec:DirectTranscriptionMethod}

Most optimal control problems need to be solved with numerical discretization schemes in practice. With a direct method, the OCP is first discretized through a transcription process, after which the resulting nonlinear programming (NLP) problem is numerically solved. Thanks to its simplicity in implementation, direct methods have become the de facto standard for solving practical optimal control problems \cite{limebeer2015faster}. 

One approach in direct methods is to solve the dynamics equations, the path constraints and the boundary conditions altogether on a discretization mesh. This is often referred to as direct collocation methods. Moreover, it can be further categorized into fixed-order $h$ methods (e.g.\ Euler, Trapezoidal, Hermite-Simpson (H-S) and the Runge-Kutta (RK) family) \cite{betts2010practical}, and variable higher-order $p$/$hp$ methods \cite{fahroo2008advances,liu2014hp}. Here, we aim to provide a high level overview, which is valid for both $h$ and $p$/$hp$ methods. 

With a mesh of size $N=\sum_1^K N^{(k)}$, the states can be approximated as

\begin{equation}
\label{eqn: LGRStateApproximation}
x^{(k)}(\tau) \approx \bar{x}^{(k)}(\tau) := \sum_{j=1}^{N^{(k)}}\mathcal{X}_j^{(k)}\mathcal{B}_{j}^{(k)}(\tau),
\end{equation}

with mesh interval $k$ $\in$ $\{1$, $\hdots$, $K\}$, $N^{(k)}$ denoting the number of collocation points  for  interval $k$, and $\mathcal{B}_{j}^{(k)}(\cdot)$ are basis functions. For classical $h$ methods, $\tau \in \mathbb{R}^{N}$ takes on values in the interval $[0,1]$ representing $[t_0,t_f]$, and $\mathcal{B}_{j}^{(k)}(\cdot)$ are chosen to be elementary B-splines of various orders. For $p$/$hp$ methods, $\mathcal{B}_{j}^{(k)}(\cdot)$ are Lagrange interpolating polynomials over the normalized time interval $\tau$ $\in$ $[-1,1]$.  We use $X_j^{(k)}$ and $U_j^{(k)}$ to represent the approximated states and inputs at collocation points, e.g.\  $X_j^{(k)}=\bar{x}^{(k)}(\tau_j^{(k)}) \in \mathbb{R}^{n}$, where $\tau_j^{(k)}$ is the $j^\text{th}$ collocation point in mesh interval~$k$.

Consequently, the optimal control problem~\eqref{eqn:OCPBolza} can be approximated by
\begin{subequations}
\label{eqn:LGRStateApproximationCost}
\begin{multline}
\min_{X,U,p,t_0,t_f}  \Phi(X_1^{(1)},t_0,X_{f}^{(K)},t_f,p)\\
+\sum_{k=1}^{K}\sum_{i=1}^{N^{(k)}} w_i^{(k)} L(X_i^{(k)},U_i^{(k)},\tau_i^{(k)},t_0,t_f,p)
\end{multline}
subject to, for $i=1,\hdots,N^{(k)}$ and $k=1,\hdots,K$,  
\begin{align}
\label{eqn:LGRStateApproximationCostDefect}
\sum_{j=1}^{N^{(k)}}\mathcal{A}_{ij}^{(k)}X_j^{(k)}+\mathcal{D}_{i}^{(k)}f(X_i^{(k)},U_i^{(k)},\tau_i^{(k)},t_0,t_f,p) = & 0 \\
\label{eqn:LGRStateApproximationPathConstraint}
c(X_i^{(k)},U_i^{(k)},\tau_i^{(k)},t_0,t_f,p)\le & 0  \\
\phi(X_1^{(1)},t_0,X_{f}^{(K)},t_f,p) =& 0
\end{align}
\end{subequations}
where $w_j^{(k)}$ are the quadrature weights for the respective discretization method chosen, $\mathcal{A}$ is the numerical differentiation matrix with $\mathcal{A}_{ij}$ the element $(i,j)$ of the matrix, and $\mathcal{D}$ a constant matrix. The discretized problem can then be solved with off-the-shelf NLP solvers, such as interior point solver \texttt{IPOPT} \cite{wachter2006implementation}.  

 The NLP solver outputs a discretized solution $\mathcal{Z} \coloneqq (X, U, p, \tau, t_0, t_f)$ as sampled data points, however it does not necessarily indicate that the solution is a discrete-time control sequence. In-between the sampled points, interpolating splines can be used to construct an approximation of the continuous-time optimal trajectory $t\mapsto \tilde{\mathcal{Z}}(t) \coloneqq (\tilde{x}(t), \tilde{u}(t), t, p)$ in accordance to the discretization method employed. The quality of the interpolated solution needs to be assured through error analysis. If necessary, modifications must be made accordingly to the discretization mesh, until the solutions obtained with the new mesh fulfills all predefined error tolerance levels (e.g.\ the absolute local error $\eta_{tol}$ and the absolute local constraint violation $\varepsilon_{c_{tol}}$).  This process is known as mesh refinement.

\section{Implementations of Rate Constraints}
\label{sec:AlgebraicVariableRateConstraints}
In many problems, constraints of the form 
\begin{align*}
\dot{u}_{L} \le \frac{du}{dt}(t) &\le \dot{u}_{U}\\
\dot{x}_{L} \le \frac{dx}{dt}(t) &\le \dot{x}_{U}
\end{align*}
may need to be implemented to restrict the rate of change for the state and/or input variables. 

\subsection{Conventional Implementation}

For input variables, a common approach is to introduce $u$ as an additional state variable, and $\nu$ as the new input with a simple bound through the dynamic equation
\begin{equation}
\label{eqn:SBDeltau}
\dot{u}(t)=\nu(t) \text{ with } \dot{u}_{L} \le \nu(t) \le \dot{u}_{U}.
\end{equation}
For rate constraints on the state variable $x$, additional path constraints are needed:
\begin{equation}
\label{eqn:SBPathf}
\dot{x}_{L} \le f(x(t),u(t),t,p) \le\ \dot{x}_{U}.
\end{equation}
For simplicity, we refer to \eqref{eqn:SBDeltau} as the \emph{add-state} implementation, and~\eqref{eqn:SBPathf} as the \emph{add-path constraint} implementation.

Unfortunately these  conventional  implementations exhibit many shortcomings. These are mainly:
\begin{enumerate}
\item The number of state variables and constraint equations are increased, resulting in a larger NLP. In addition, the index of the DAE (differential-algebraic equations) of the transcribed problem may also increase, leading to a problem that is often more difficult to solve numerically.  
\item When \eqref{eqn:SBDeltau} is used, singular arcs may occur and affect the solution quality. This can occur if the original control input $u$ appears nonlinearly in the Lagrange cost or other system dynamics and the new control $\nu$ appears linearly instead. 
%\item For some higher order discretization schemes, ensuring state rate constraint satisfaction with \eqref{eqn:SBPathf} at collocation points is not sufficient for guaranteeing constraint compliance inside the intervals in between collocation points. Consequently, additional mesh refinement criteria may need to be specified to bring the constraint violation to an acceptable level.
\end{enumerate}

Detailed justification for the first point can be found in \cite{betts2010practical}. Here, we focus the discussions on the issue of singular arcs, which also falls with a much larger concern regarding the quality of the solution. 

Generally speaking, a singular arc is an interval in the OCP solution where the optimality conditions yield no information about the optimal control function. Precise mathematical definitions and an analysis of singular arcs can be found in \cite{hermes1963nonlinear} and \cite{straeter1970singular}. For ease of demonstration, consider the following OCP, which is simplified but still sufficiently general,  with $x_1 \in \mathbb{R}$ the state variable, and $u_1 \in \mathbb{R}$ the control input. As per the conventional approach, the rate constraint on the original control input is implemented with a new control input~$\nu_1 \in \mathbb{R}$: 

\begin{subequations}
\label{eqn:OCPGeneralSingular}
\begin{equation}
%\begin{split}
\min_{x_1,u,\nu} \int_{0}^{t_f} g_1(x_1(t)) + g_2(x_1(t),u_1(t)) dt
%\end{split}
\end{equation}
subject to
\begin{align}
\dot{x_1}(t)=g_3(x_1(t),u_1(t)) \quad &\forall t \in [t_0,t_f]\\
\label{eqn:OCPGeneralSingularAddDynamics}
\dot{u}_1(t)=\nu_1(t) \quad &\forall t \in [t_0,t_f]\\
\label{eqn:OCPGeneralSingularAddSB}
\dot{u}_{1_{L}}  \le \nu_1(t) \le \dot{u}_{1_{U}} \quad &\forall t \in [t_0,t_f].
\end{align}
\end{subequations}
 Here we follow the same hypotheses as in \cite{Vinter2013}: $g_1$, $g_2$ and $g_3$ are continuous, continuously differentiable for all $u_1 \in \mathcal{U}$, and Lipschitz in $x_1$. Also, the admissible control set $\mathcal{U}$ is assumed to be a bounded set in some Euclidean space.

\begin{proposition}
\label{prop: OCPGeneralSingular}
If the OCP \eqref{eqn:OCPGeneralSingular} has a linear objective and dynamics with respect to the original control input $u_1$, i.e.\  if $g_2$ and $g_3$ are both functions that only have strictly linear input (i.e. in the form of $g_{\iota}(x_1(t),u_1(t))=\tilde{g}_{\iota}(x_1(t))+c_{st}u_1(t)$, with $c_{st}$ a constant),  the resulting optimal control $\nu_1^{\ast}$ will not contain a singular arc. However, if $u_1$ appears nonlinearly in the objective and/or dynamics, i.e.\ if $g_2$ and/or $g_3$ are arbitrary nonlinear functions,  there exists problems where singular arcs will occur  for some intervals of the solution. % when solved using the optimality conditions.
\end{proposition}
\begin{proof}
First, we formulate the Hamiltonian of the system, with~$\lambda(t)$ the costate of the dynamics
\begin{align}
\label{eqn:OCPGeneralSingularH}
\begin{split}
&H(x_1(t), u_1(t),  \lambda_{x_1}(t), \lambda_{u_1}(t), \nu_1(t)) := g_1(x_1(t))   \\
& + g_2(x_1(t),u_1(t)) + \lambda_{x_1}(t)g_3(x_1(t),u_1(t)) + \lambda_{u_1}(t)\nu_1(t).
\end{split}
\end{align}
From Pontryagin's minimum principle, we know that if the state and costate are optimal, the optimal control $\nu_1^\ast$  minimizes the Hamiltonian, i.e.
\begin{equation}
\label{eqn:OCPGeneralSingularUStar}
\nu_1^{\ast}(t) \in \arg \min_{v} H(x_1^{\ast}(t), u_1^{\ast}(t), \lambda_{x_1}^{\ast}(t), \lambda_{u_1}^{\ast}(t), v).
\end{equation}
Substituting the Hamiltonian  \eqref{eqn:OCPGeneralSingularH} in \eqref{eqn:OCPGeneralSingularUStar} yields the optimal control
\begin{equation}
\label{eqn:OCPGeneralSingularUStar2}
\nu_1^{\ast}(t) = \begin{cases}
 \dot{u}_{1_{U}} & \text{if }\lambda_{u_1}^{\ast}(t) < 0 \\
 ? & \text{if }\lambda_{u_1}^{\ast}(t) = 0 \\
 \dot{u}_{1_{L}} & \text{if }\lambda_{u_1}^{\ast}(t) > 0 
\end{cases}
\end{equation}

We first show that this implementation yields a singular arc free  solutions with linear OCPs. From the first order necessary conditions for optimality we have $\dot{\lambda}_{u_1}(t)=-\frac{\partial H}{\partial u_1}$. Thus,  if both $g_2$ and $g_3$ only contain input terms in the form of $c_{st}u_1(t)$, $\dot{\lambda}_{u_1}(t)$ will then be constant and $\lambda_{u_1}(t)$ will be a linear straight line. Relating this to \eqref{eqn:OCPGeneralSingularUStar2} we can see that the optimal control will exhibit bang-bang behaviour with at most one switch depending on the crossing of $\lambda_{u_1}(t)$ with the $x$ axis. Therefore, the solution is free of singular arcs if the objective and dynamics with respect to the original control input $u_1$ are all linear.  This is the reason why in linear optimal control problems, the conventional implementation can be used without issues.

 To show that this implementation will suffer from singular arc problems when nonlinear OCPs are considered, we assume that $g_2$ and/or $g_3$ are now arbitrary nonlinear functions. Thus, $\lambda_{u_1}(t)$ can be a function of any shape and the optimal control will not be uniquely defined on intervals where $\lambda_{u_1}(t)=0$, a.k.a.\ the singular arc.  The problem in Section \ref{subsec:SOSExample} is an example where such an issue arises. 
\end{proof}

 For a direct collocation method to yield the correct solution for singular control problems, one might have to use a multi-phase formulation and additionally impose the singular arc condition specifically on the phases with singular control. For example, if one takes the same approach (as in the proof) for the example problem in Section \ref{subsec:SOSExample}, the condition for singular control to occur is when $\lambda_{x_2}=0$. Repetitively taking time derivatives of this equation would yield the singular arc condition $u(t)=x_1(t)$. We again note that this approach requires analytical derivations and would become increasingly challenging for complex real-world problems. 

An ad-hoc method sometimes used in practice for dealing with the singular arc is to augment the original objective with an additional regularization term (e.g.\ in \cite{Bonnans2008}), often in the form of $\rho ||\nu||_{\mathcal{L}_1}$ or $\rho ||\nu||_{\mathcal{L}_2}^2$. With relatively large values of the penalty weight $\rho$, the fluctuations on the singular arc can be suppressed, but at the cost of obtaining sub-optimal trajectories. To get closer to the optimal from this point, the problem may need to be repetitively solved with the penalty weight gradually reduced.

\subsection{On-mesh Implementation}
\label{subsec:onmeshimplementation}
To mitigate the above-mentioned shortcomings, a method is proposed to directly impose algebraic rate constraints for input variables on the discretization grid. Based on previous work \cite{betts2010practical}, we generalize this on-mesh approach for all collocation methods ($h$, $p$ and $hp$ type), as well as for state variables. 

Since the treatment for state variables $x$ and input variables $u$ are similar, for simplicity we will use $z$ to represent the variable on which the rate constraints are imposed. If $Z_i$ represents the discretized version of $z$ at time instance~$i$, then the numerical differentiation of $z$ at that grid point ($Z_i'$) can be calculated using $s$-point finite difference approximations, with $s$ the number of data points in the interval (including endpoints). See Table \ref{tab:NumericalDifferentiationSchemes} for the formulations of some of the most commonly used discretization methods, with $\Delta \tau_i = \tau_{i+1}-\tau_i$, $\Delta t = t_f-t_0$, and $\mathcal{A}_{LGR}$ is the Legendre-Gauss-Radau (LGR) differentiation matrix. Details on the determination of the numerical differentiation equations are available in \cite{fornberg1988generation}.

\begin{table}[b]
	\begin{center}
	\caption{Numerical differentiation schemes}
	\label{tab:NumericalDifferentiationSchemes}
		\begin{tabular}{c|c|c}
		 \multirow{2}{*}{\textbf{Method}} & \textbf{No. of Data}  & \multirow{2}{*}{\textbf{Numerical Differentiation}}\\
		 & \textbf{Points} ($r$)& \\
		\hline Trapezoidal & 2 &\multirow{2}{*}{$Z'_i=\frac{Z_{i+1}-Z_i}{\Delta t \Delta \tau_i}$} \\
		($h$) & (equal spaced) &\\
		\hline  & &\multirow{2}{*}{$Z'_i=\frac{-3Z_i+4Z_{i+1/2}-Z_{i+1}}{\Delta t \Delta \tau_i}$}\\
		 & \\
		 Hermite & 3 &\multirow{2}{*}{$Z'_{i+1/2}=\frac{Z_{i+1}-Z_{i}}{\Delta t \Delta \tau_i}$}\\
		 Simpson ($h$)& (equal spaced) &\\
		 & &\multirow{2}{*}{$Z'_{i+1}=\frac{Z_{i}-4Z_{i+1/2}+3Z_{i+1}}{\Delta t \Delta \tau_i}$}\\
		 & \\
		\hline LGR & N+1&\multirow{2}{*}{$Z'_{1:N+1}=\frac{2}{\Delta t}\mathcal{A}_{LGR}Z_{1:N+1}$} \\
		 ($p$/$hp$) & (LGR bases) &\\
		\hline
		\end{tabular} 
	\end{center}
\end{table}

Note that for $p$/$hp$ methods, the numerical differentiation for all grid points on the polynomial ($i=1, \hdots, N^{(k)}$) are obtained altogether. It is also worth mentioning that if Legendre-Gauss-Radau (LGR) collocation is used, the end-point value for the control ($U_{N+1}^{(K)}$) may need to be approximated.

It is then straightforward to implement the rate constraints as linear constraints
\begin{subequations}
\label{eqn:RateConstraintLinear}
\begin{align}
\dot{z}_{L} - Z'_i &\le 0\\
 Z'_i - \dot{z}_{U} &\le 0
\end{align}
\end{subequations} 
for all possible values of $i$. This approach will be referred to as the \emph{on-mesh} implementation.

The on-mesh implementation of rate constraints has several benefits in comparison to the conventional add-state and add-path constraint approaches. Firstly, we compare the solution quality in terms of singular arcs. A challenge arises here, since the singular arc problem is commonly analyzed with the original OCP \eqref{eqn:OCPBolza}, but~\eqref{eqn:RateConstraintLinear} is a discretized formulation that does not have the continuous form.  Therefore, we do not yet have a mathematically rigorous proof that the on-mesh implementation will be singular-arc free using Pontryagin's minimum principle. However, one would observe that, without introducing a dynamic constraint in the form of $\dot{u}(t)=\nu(t)$, the singular control situation as described in Proposition \ref{prop: OCPGeneralSingular} will not occur, at least not in the same way. In addition, in our computational experience, we had not encountered a single case where the on-mesh implementation causes a singular-arc free problem to become singular. In contrast, the on-mesh rate constraint method was able to convert many well-known singular control problems to be singular arc free ones, with one example demonstrated in Section \ref{subsec:SOSExample}.

Another major advantage in comparison to the conventional implementation is regarding the computational cost. For systems with nonlinear dynamics, rate constraints on state variables implemented through \eqref{eqn:SBPathf} will be nonlinear path constraints relating different state variables at the same time instance. Thus their Jacobian and Hessian contributions can make the solution of the OCP computationally demanding. In contrast, on-mesh implementation of rate constraints with~\eqref{eqn:RateConstraintLinear} are linear constraints, with no contribution to the Hessian. 

In addition, note that the rate constraints \eqref{eqn:RateConstraintLinear} only depend on the numerical differentiation schemes. Thus, once a discretization scheme for the OCP has been chosen, and the corresponding discretization mesh has been determined, the Jacobian contributions of the rate constraint equations can be pre-computed during the transcription process. Therefore, although the NLP dimension increases more rapidly with the on-mesh implementation as shown in Table~\ref{tab:ProblemSizeInput} and~\ref{tab:ProblemSizeState}, the computational complexity for obtaining the derivative information with respect to the rate constraint equations is actually lower than the conventional approach. Altogether, the computational advantages can be rather significant, as demonstrated with the example problem.

\begin{table}[tb]
	\begin{center}
	\caption{Contribution to the NLP dimensions with an $N$-point mesh for different rate constraint implementations on input variables}
	\label{tab:ProblemSizeInput}
		\begin{tabular}{c|c|c}
		 & \textbf{add-state} &  \textbf{on-mesh}  \\
		 \hline   2 point, & $N$ linear  & $2N$ linear\\
		 collocated & inequality constraints & inequality constraints\\
		  (Trapezoidal) & (defect constraints) & (pre-computed) \\
		 \hline  3 point, & $2N-1$ linear & $4N-4$ linear \\
		 collocated & inequality constraints & inequality constraints\\
		  (H-S) & (defect constraints) & (pre-computed) \\
		  \hline  $p$ point, & $p(N-1)$ linear  & $2p(N-1)$ linear \\
		 collocated & inequality constraints & inequality constraints\\
		  (LGR) & (defect constraints) & (pre-computed) \\
		\hline
		\end{tabular} 
	\end{center}
\end{table}

\begin{table}[tb]
	\begin{center}
	\caption{Contribution to the NLP dimensions with an $N$-point mesh for different rate constraint implementations on state variables}
	\label{tab:ProblemSizeState}
		\begin{tabular}{c|c|c}
		 & \textbf{add-path constraint} &  \textbf{on-mesh}  \\
		 \hline   2 point, & $2N$ nonlinear  & $2N$ linear \\
		 collocated & inequality constraints & inequality constraints\\
		  (Trapezoidal) & (path constraints) & (pre-computed) \\
		 \hline  3 point, & $4N-2$ nonlinear  & $6N-6$ linear \\
		 collocated & inequality constraints & inequality constraints\\
		  (H-S) & (path constraints) & (pre-computed) \\
		  \hline  $p$ point, & $2p(N-1)$ nonlinear  & $2p(N-1)$ linear \\
		 collocated & inequality constraints & inequality constraints\\
		  (LGR) & (path constraints) & (pre-computed) \\
		\hline
		\end{tabular} 
	\end{center}
\end{table}

A remark is appropriate when comparing the on-mesh implementation in Table~\ref{tab:ProblemSizeState} to Table~\ref{tab:ProblemSizeInput}: For Hermite-Simpson discretization, specifically, the increase in the size of the NLP for implementation on input variables is less than that on state variables. This is because, when the control $u$ is discretized as a quadratic function of time, the rate of change w.r.t.\ time ($\dot{u}$) is linear, thus extreme values only occur at the end-points of each interval ($U_i$ and $U_{i+1}$). In this special case only, the rate constraints relating to the middle points ($U'_{i+1/2}$) can be neglected. %Another situation where  the number of constraint equations may be reduced is if $\dot{z}_{L}=\dot{z}_{U}$.

\section{Example Problems}
\label{sec:ExampleProblems}
 The problem of singular arcs is often demonstrated with toy problems in the literature (e.g.\ the first example), as they are much more illustrative and free from influence of other factors. However, this common practice often results in it being neglected by engineers working on complex problems. To show that it really matters, a second real-world example is also presented here to demonstrate the acclaimed benefits of the on-mesh rate constraint implementation in terms of solution quality and computational efficiency.

\subsection{Second order singular regulator}
\label{subsec:SOSExample}

First, we consider a simple regulator problem originally presented in \cite{aly1978computation}. This can be considered as regulation control of a double integrator system, with a constraint on the acceleration. 
\begin{subequations}
\begin{equation*}
%\begin{split}
\min_{x_1,x_2,u} \int_{0}^{5} x_1^2(t) + x_2^2(t) dt
%\end{split}
\end{equation*}
subject to
\begin{equation*}
\dot{x_1}(t)=x_2(t), \quad \dot{x_2}(t)=u(t) \in [-1,1] \quad \forall t \in [0,5].
\end{equation*}
\end{subequations}
In this original OCP formulation, the optimal control is in the form of \textit{bang-singular}. Figure \ref{fig:CompInputSOS} shows that a numerical implementation of this OCP using direct collocation would yield fluctuations and ringing phenomena for solutions at collocation points, as well as for the trajectories in-between.

However, upon noticing that the second differential equation is equivalent to the add-state implementation of a rate constraint $-1 \le \dot{x_2}(t) \le 1$, we can remove that differential equation from the OCP and use the on-mesh rate constraint method instead. As illustrated in the figure, this approach successfully yields a stable and accurate solution in correspondence to the reference (analytical) optimal input trajectory. In other words, the proposed on-mesh rate constraint implementation has successfully converted the classical second order singular regulator problem into a singular-arc-free formulation. 

\begin{figure}[tb]
\begin{center}
\includegraphics[width=\columnwidth]{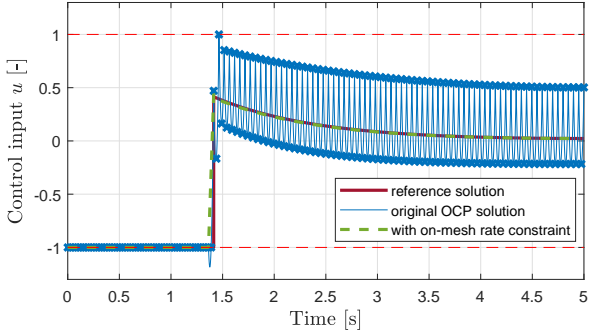}    % The printed column width is 8.4 cm.
\caption{ Comparison of obtained control input trajectories for the second order singular regulator problem (H-S discretization with 199 collocation points (100 mesh nodes), crosses represent the collocation points) } 
\label{fig:CompInputSOS}
\end{center}
\end{figure}

\subsection{Aircraft go-around in the Presence of Windshear}
\label{subsec:AircraftExample}

Based on the developments by \cite{miele1988optimal,bulirsch1991abort1,bulirsch1991abort2}, a problem is presented in \cite{betts2010practical} where the aircraft needs to stay as high above the ground as possible after encountering a severe windshear during landing. %The implementation used here contains slight modifications to the windshear modelling, with the exponential functions approximated by piecewise polynomial functions.

 Details about the system dynamics, variable simple bounds, boundary conditions, aerodynamic modelling, as well as static parameter values, are available in the references above. The problem has state variables being the horizontal distance $d$,  the altitude $h$,  the true airspeed $V$, the flight path angle $\gamma$, and the input variable being the angle of attack $\alpha$ . We also emphasise that the problem requires the implementation of a rate constraint $\lvert \dot{\alpha}(t)\rvert \le 3$ deg/s.  

To avoid discontinuities and to assist convergence, a static optimization parameter $h_{min}$ is introduced to represent the minimum altitude. The objective can then be expressed as $\Phi(x(t_0),t_0,x(t_f),t_f,p):=-h_{min}$ together with a new path constraint $h(t) \ge h_{min}$.

%The following simple bounds on some of the state variables
%\begin{align}
%\label{eqn:ConstraintAircraftGPW}
%\begin{split}
%0 \le d(t) \le 10000 \text{ [ft],} \quad & \quad 0 \le h(t) \le 1000 \text{ [ft],} \\ 
%0 \le V(t) \le \infty \text{ [ft/s],}\quad & \quad -\infty \le \gamma(t) \le \infty \text{ [deg],} \\
% -17 \le  \alpha(t) \le 17 \text{ [deg],} \quad & \quad -3 \le \dot{\alpha}(t) \le 3 \text{ [deg/s],}
%\end{split}
%\end{align}
%are imposed together with the boundary conditions
%\begin{align*}
%\label{eqn:BCAircraftGPW}
%\begin{split}
%d(0)=0\text{ [ft],} \quad h(0)=6&00\text{ [ft],} \quad V(0) = 239.7\text{ [ft/s],} \\
%\gamma(0)=-2.25\text{ [deg],} &\quad \alpha(0)=7.35\text{ [deg],}\\
%t_f=40\text{ [s],} \quad d(t_f) &= \text{free,} \quad h(t_f)=\text{free,}  \\
% V(t_f) =  \text{free,} \quad \gamma(t_f)=&7.43 \text{ [deg],} \quad \alpha(t_f)=\text{free.}
%\end{split}
%\end{align*}

Figure \ref{fig:InputRCSolution} illustrates the solution to the problem using Hermite-Simpson discretization. 
\begin{figure}[t]
\begin{center}
\includegraphics[width=\columnwidth]{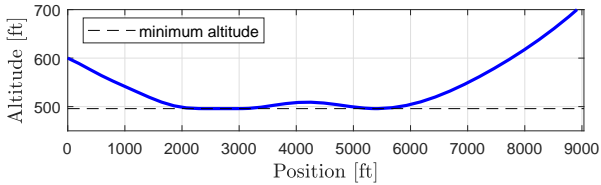}    % The printed column width is 8.4 cm.
\caption{Solution to the aircraft go-around in the windshear problem, with input rate constraints} 
\label{fig:InputRCSolution}
\end{center}
\end{figure}
% * <e.kerrigan@imperial.ac.uk> 2018-03-21T12:45:47.816Z:
% 
% Maybe format Figure 1 as in Figure 2?
% 
% ^ <yuanbonie@outlook.com> 2018-03-21T15:08:45.367Z:
% 
% I have tried that at start, but with 3 sub-figures, it looks too flattened out. Will it be a problem to keep it like this?
%
% ^ <e.kerrigan@imperial.ac.uk> 2018-03-21T18:01:57.945Z:
% 
% You can try and do what Omar does, which is to put all the variables on the same plot and scale them and colour them. This will save a lot of space - you could try this with your other figures as well.
% 
% ^.
All figures presented in this paper are the outcome of a mesh refinement scheme that minimizes errors to the tolerance as specified in Table \ref{tab:MRScriteria}. 

\begin{table}[b]
\small
\begin{center}
\caption{Mesh refinement criteria}\label{tab:MRScriteria}
\begin{tabular}{c|c|c|c|c|c|c}
 & $d$  & $h$  & $v$ & $\gamma$  & $\alpha$  & Path Constraint\\
  & [ft] & [ft] & [ft/s] & [deg] & [deg] & [m]\\\hline
$\eta_{tol}$ & 1 & 0.5 & 0.1 & 0.5 & 0.5 & -\\ \hline 
$\varepsilon_{g_{tol}}$ & 1 & 0.5 & 0.1 & 0.5 & 0.5 & $1\times10^{-5}$\\ 
\hline 
\end{tabular}
\end{center}
\end{table}

%It is important to note that, although different implementations of rate constraints can influence the computational performance, the singular arc behaviour will not have noticeable effects on the solution to the state variables. Thus the solution in Figure \ref{fig:InputRCSolution} should be obtainable regardless of the discretization method and the rate constraint implementation.  

\subsubsection{Implementation of Rate Constraints for Input Variables}
 Constraint $\lvert \dot{\alpha}(t)\rvert \le 3$ applies directly on the rate of change for the control input $\alpha$.  Using the conventional approach, $\alpha$ can be treated as an additional state variable, and $\nu$ introduced as the new control input with the dynamics 
\begin{equation}
\dot{\alpha}(t) =  \nu(t).
\end{equation}
Thus the rate constraints for $\alpha$ can be implemented as simple bounds on $\nu$: $-3 \le  \nu(t) \le 3$ [deg/s].

As mentioned earlier, due to the fact that the original control input $\alpha$ appears nonlinearly in the system, whereas the new input $\nu$ appears linearly, singular arc behaviour can occur, which is  shown in Figure \ref{fig:udothInputs}, with large fluctuations in the solution. In contrast, when the rate constraints are directly implemented on the discretization mesh instead (Figure \ref{fig:ARChInputs}), the optimal control input trajectory can be obtained with little ambiguity.

\begin{figure}[t]
    \centering
    \begin{subfigure}[b]{0.45\textwidth}
        \centering
        \includegraphics[width=\textwidth]{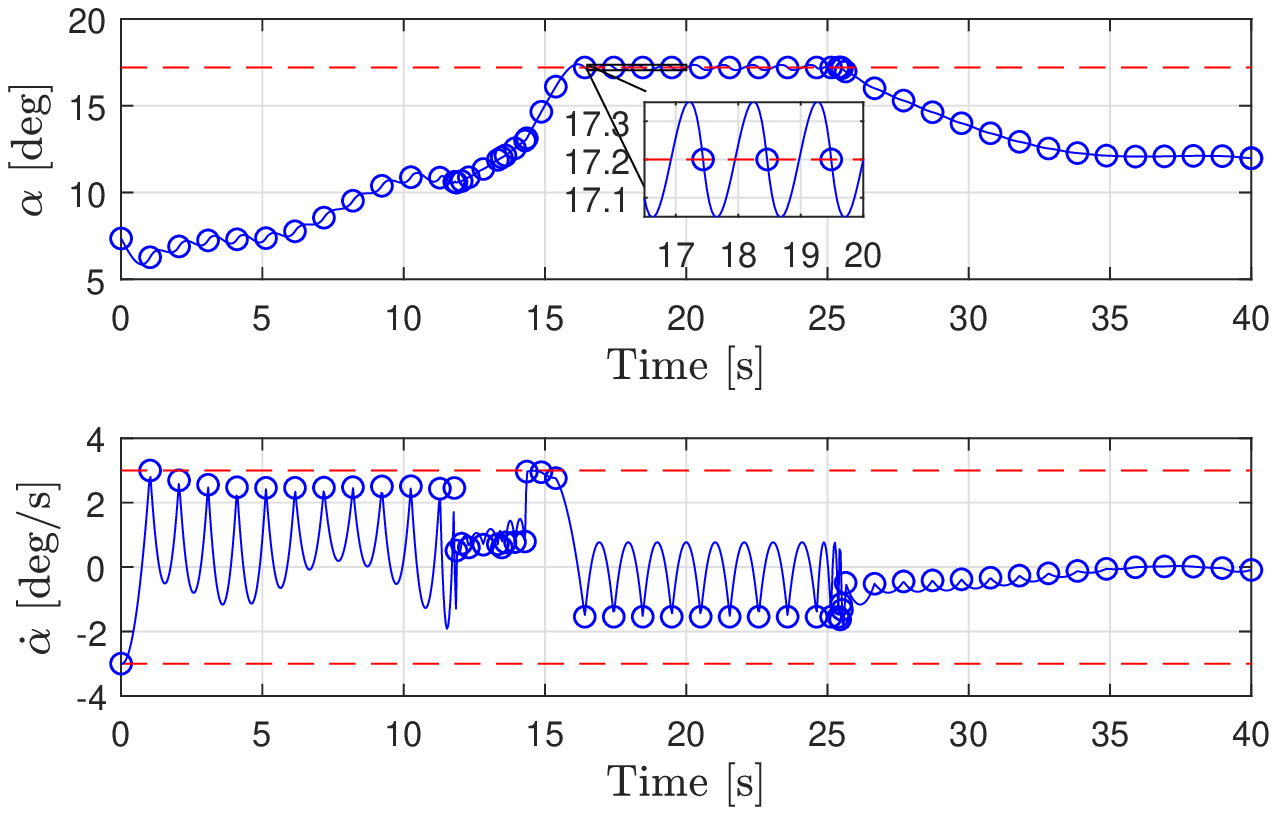}
    	\caption{implemented with additional state variable}
    	\label{fig:udothInputs}
    \end{subfigure}
	\begin{subfigure}[b]{0.45\textwidth}
        \centering
        \includegraphics[width=\textwidth]{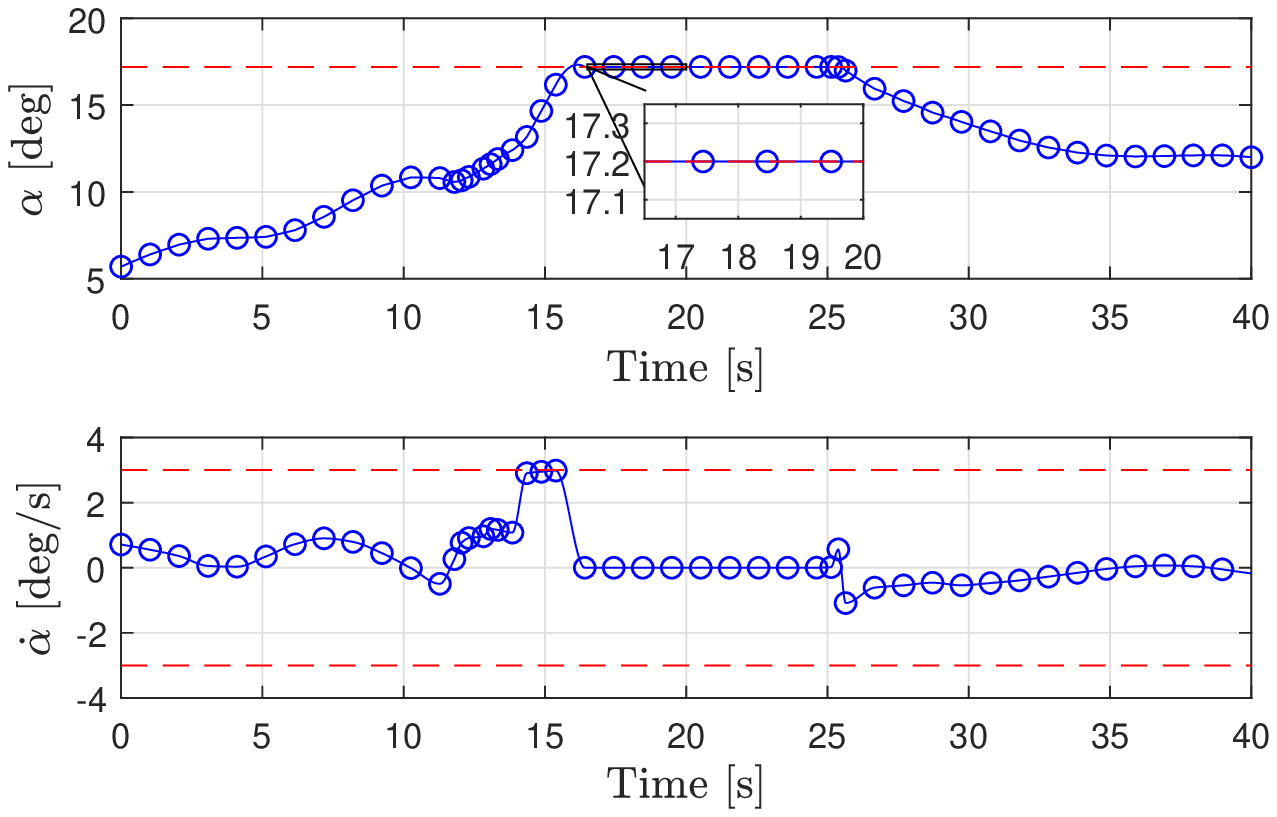}
    	\caption{direct implemention on the mesh}
    	\label{fig:ARChInputs}
    \end{subfigure}
    \caption{Control input for the solution to the aircraft go-around in the windshear problem, with different implementations for input rate constraints (H-S discretization, circles represent mesh points)}
\label{fig:hInputs}
\end{figure}

%\begin{figure}[t]
%	\begin{center}
%		\subfigure[implemented with additional state variable]{\label{fig:udothInputs}\includegraphics[width=\columnwidth]{images/udothInputs.eps}}\hspace{0.001\textwidth}                
%		\subfigure[direct implemention on the mesh]{\label{fig:ARChInputs}\includegraphics[width=\columnwidth]{images/ARChInputs.eps}}\hspace{0.001\textwidth}                
%	\caption{Control input for the solution to the aircraft go-around in the windshear problem, with different implementations for input rate constraints (Hermite-Simpson discretization, circles represent mesh points)}
%		\label{fig:hInputs}
%	\end{center}
%\end{figure}

Comparing the solutions from the two implementations, it is interesting to observe that, although the integrated values (i.e.\ angle of attack) along the singular arc solution at the collocation points are generally the same, the interpolated trajectory from the add-state method is actually distorted by the fluctuations of its rate values.

With the LGR orthogonal collocation method, improvements are relatively minor. Because the end-point value for the control input is only approximated (not a collocation point), the errors have distortion effects on all previous points of the polynomial (Figure \ref{fig:ARChpInputs}). On the other hand, thanks to this extra level of continuity imposed by higher order polynomials, the problem of singular arc behaviour is far less pronounced with the conventional add-states implementation (Figure \ref{fig:udothpInputs}), when compared to $h$ type discretization methods. 

\begin{figure}[t]
    \centering
    \begin{subfigure}[b]{0.45\textwidth}
            \centering
            \includegraphics[width=\textwidth]{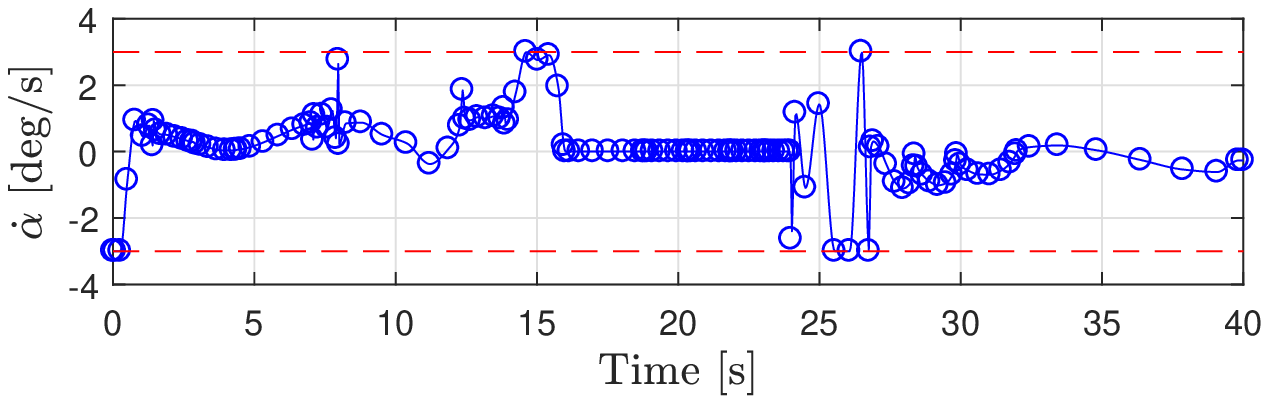}
    \caption{implemented with additional state variable}
    \label{fig:udothpInputs}
    \end{subfigure}
\begin{subfigure}[b]{0.45\textwidth}
            \centering
            \includegraphics[width=\textwidth]{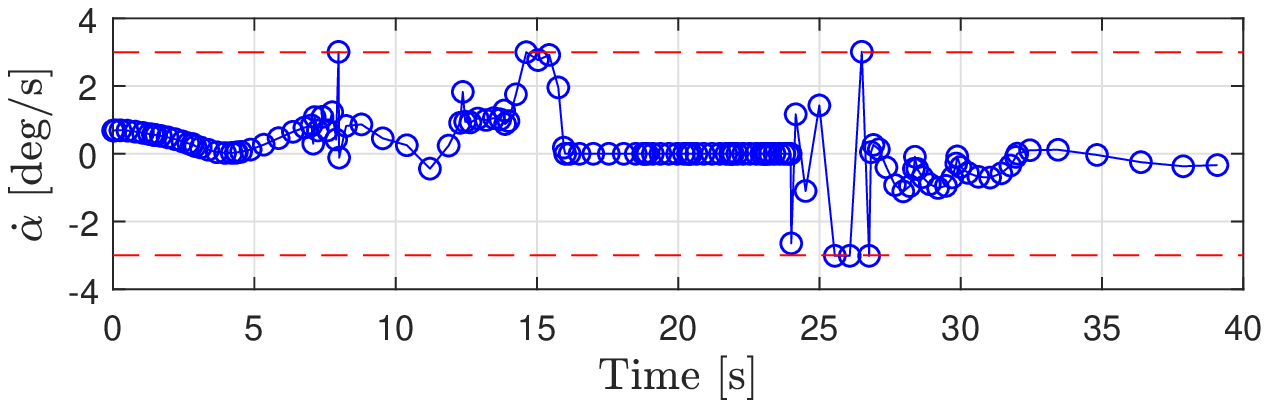}
    \caption{direct implemention on the mesh}
    \label{fig:ARChpInputs}
    \end{subfigure}
    \caption{Control input for the solution to the aircraft go-around in the windshear problem, with different implementations for input rate constraints (LGR discretization, circles represent collocation points)}
\label{fig:hpInputs}
\end{figure}

The results regarding computation times presented in Figure~\ref{fig:CompTimeInputRate} and~\ref{fig:CompTimeStateRate} were all obtained on an Intel Core i7-4770 computer, running 64-bit Windows 10 with Matlab 2017a. The OCPs were transcribed into NLP problems using the optimal control software ICLOCS2 \cite{ICLOCS2} and solved with the NLP solver IPOPT compiled with the sparse linear solver MA57 \cite{duff2004ma57}. The computation times are the averages of 10 independent runs, all starting with a very rough initial guess obtained using linear interpolation of initial and terminal conditions. 

%In all test cases the NLP solver converged quickly, with around 20 iterations for a relatively coarse mesh, and about 30 iterations for dense grids.

%Figure \ref{fig:CompTimeInputRate} compares the computation performance for different rate constraint implementations with both $h$ and $hp$ discretization methods. 

From Figure~\ref{fig:CompTimeInputRate} it can be seen that, for the computation time per iteration, the on-mesh implementation saw a slight advantage in comparison to the conventional approach. This is because the on-mesh implementation explicitly exploits the fact that the linear rate constraints have no contribution to the Hessian, and the contributions to the Jacobian are constants and can be pre-computed. The scale of the benefit also grows with the size of the mesh, from about 5\% for a coarse mesh to around 10\% for the dense mesh. 

\begin{figure}[tb]
\begin{center}
\includegraphics[width=\columnwidth]{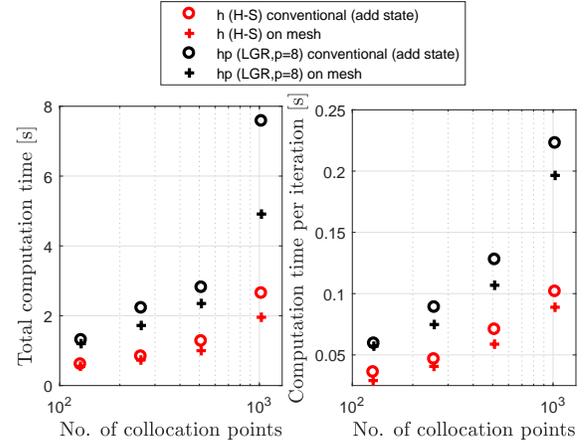}    % The printed column width is 8.4 cm.
\caption{Comparison of computational performance, with input rate constraints} 
\label{fig:CompTimeInputRate}
\end{center}
\end{figure}

Figure~\ref{fig:Regularization} presents the computation performance of the problem when regularized with an additional $\rho ||\nu||_{\mathcal{L}_2}^2$ term. It can be seen that a relative large penalty weight is required to suppress the singular arc fluctuations, but with a larger $\rho$ the result diverges quickly from the optimal. Also note that for a single solve with regularization, the norm of angle of attack rate ($||\dot{\alpha}^{\ast}-\dot{\alpha}||_{\mathcal{L}_2}$) never reaches the accuracy level obtained by the on-mesh implementation with the same discretization mesh. Thus, to obtain a good solution, $\rho$ needs to be gradually reduced, making the process complicated and computationally inefficient --- it is also difficult to guarantee solution quality.

\begin{figure}[tb]
\begin{center}
\includegraphics[width=\columnwidth]{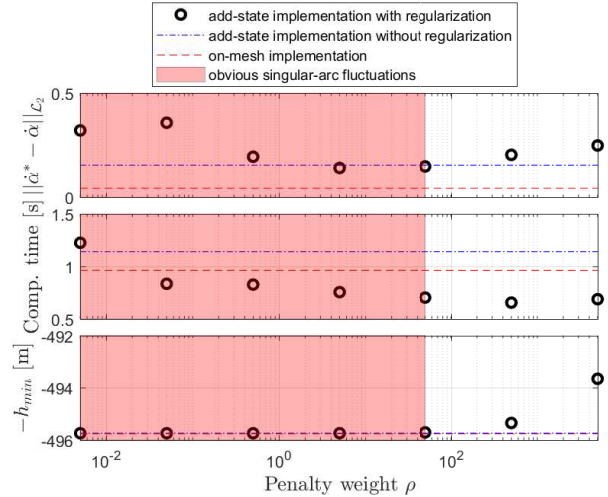}    % The printed column width is 8.4 cm.
\caption{Solution of the regularized problem with different penalty weights. (H-S collocation with 79 collocation points (40 mesh nodes); reference solution $\dot{\alpha}^{\ast}$ obtained using a very dense mesh)} 
\label{fig:Regularization}
\end{center}
\end{figure}

\subsubsection{Implementation of Rate Constraints for State Variables}

We will additionally impose a rate constraint for the velocity state as $-5 \le \dot{V}(t) \le 5 $ [ft/s$^2$]. With this new formulation, the minimum altitude achievable is slightly lower.

From Figure \ref{fig:CompTimeStateRate}, it is obvious that the two methods are not computationally comparable. Due to the reasons explained in the end of Section~\ref{subsec:onmeshimplementation}, although the increase in NLP dimension is higher for the on-mesh implementation, the resulting (larger) NLP problems with linear rate constraints are actually much easier to solve. Consequently, regardless of the discretization method, the computation time per iteration recorded for the on-mesh implementations are all significantly (more than 30\%) lower than the conventional method. 

\begin{figure}[tb]
\begin{center}
\includegraphics[width=\columnwidth]{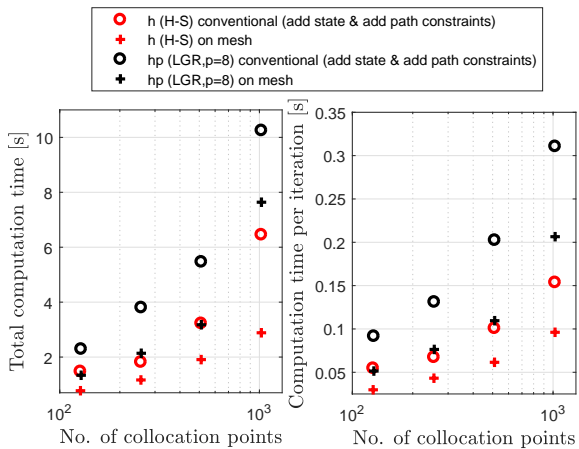}    % The printed column width is 8.4 cm.
\caption{Comparison of computational performance, with state and input rate constraints} 
\label{fig:CompTimeStateRate}
\end{center}
\end{figure}

\section{Conclusions}
\label{sec:Conclusion}
Through both the mathematical analysis and a computation example, we demonstrated that mathematically equivalent formulations for rate constraints on state and input variables may not have the same solution quality and computational complexity in numerical implementations. For all collocation methods tested, and for both state and input variables, the proposed approach to directly implement rate constraints on the discretization mesh appears to be an attractive alternative for nonlinear optimization based control problems. 

In contrast to conventional approaches, the proposed method is numerically verified to not introduce singular control arcs, a phenomena which may lead to severe distortions and fluctuations in the optimal control trajectories. Additionally, this on-mesh implementation replaces the additional dynamic equations and nonlinear path constraints in conventional implementations with linear rate equations. Thus, there is no contribution to the Hessian and the contribution to the Jacobian can be pre-computed, enabling faster iterations. Based on observations, the scale of reduction in computation time seems to grow quite quickly with the increase in mesh size (number of collocation points), making the method especially suitable for the solution of large scale problems.  A possible downside of the proposed scheme is that the method cannot be directly implemented in most of the existing OCP packages through their current interfaces. However, we believe that this implementation can be easily added by the authors of the software, and ensure that all computational benefits are fully realized. 

\bibliography{main} 
%\bibliography{GC2017} 
\bibliographystyle{ieeetr}

\end{document}